\documentclass[a4paper,reqno]{amsart}

\usepackage{amsfonts}
\usepackage{mathrsfs}
\usepackage{graphicx}
\usepackage[normalem]{ulem}
\usepackage{url}
\usepackage{amsthm,amsfonts,amssymb,amsmath,esint,amscd,latexsym,amsxtra}
\usepackage[colorlinks=true, urlcolor=blue,bookmarks=true,bookmarksopen=true,
citecolor=blue]{hyperref}
\usepackage{amscd}
\usepackage{enumerate,bbm}

\usepackage{fullpage}
\usepackage[all]{xy}
\usepackage[OT2,T1]{fontenc}
\usepackage{xspace}

\newcommand{\BC}{{\mathbb {C}}} 
 
\newcommand{\BG}{{\mathbb {G}}}

\newcommand{\CK}{{\mathcal {K}}}

 \newcommand{\CT}{{\mathcal {T}}}

 \newcommand{\fh}{{\mathfrak{h}}}

 \newcommand{\fX}{{\mathfrak{X}}}

\newcommand{\ad}{{\mathrm{ad}}}

 \newcommand{\EP}{{\mathrm{EP}}}

 \newcommand{\GL}{{\mathrm{GL}}}

\newcommand{\Hom}{{\mathrm{Hom}}}

\newcommand{\Rep}{{\mathrm{Rep}}}

\newcommand{\Res}{{\mathrm{Res}}}

\newcommand{\St}{{\mathrm{St}}}

\newcommand{\Vol}{{\mathrm{Vol}}}

\newcommand{\Ext}{{\mathrm{Ext}}}

\newcommand{\matrixx}[4]{\begin{pmatrix}
		#1 & #2 \\ #3 & #4
\end{pmatrix} }        

\newcommand{\wh}[1]{{\widehat {#1}}}

\newcommand{\ov}[1]{{\overline{#1}}}

\newcommand{\sk}{\medskip}

\newcommand{\bs}{\backslash}

\newcommand{\s}{\sk\noindent}

\makeatletter\def\varW@#1#2{%
\vtop{\m@th\ialign{##\cr
\hfil$#1 \mathrm{colim} $\hfil\cr
\noalign{\nointerlineskip\kern1.5\ex@}#2\cr
\noalign{\nointerlineskip\kern-\ex@}\cr}}
}
\makeatother
\makeatletter
\def\colim{%
\mathop{\mathpalette\varW@{}}\nmlimits@
}\makeatother

\theoremstyle{plain}
\newtheorem{thm}{Theorem}[section] \newtheorem{cor}[thm]{Corollary}

\newtheorem{lem}[thm]{Lemma}  \newtheorem{prop}[thm]{Proposition}

\theoremstyle{remark} \newtheorem{remark}[thm]{Remark}
\theoremstyle{definition} 
\theoremstyle{definition} \newtheorem{example}[thm]{Example} 
\newtheorem{defn+lem}[thm]{Definition and Lemma}

\numberwithin{equation}{section}

\newcommand*{\sheafhom}{\mathrm{H}\kern -.5pt om}

\begin{document}
\title{Higher Ext-groups in the triple product case}
\date{}

\author{Li Cai}
\address{Academy for Multidisciplinary Studies\\
Beijing National Center for Applied Mathematics\\
Capital Normal University\\
Beijing, 10048, People's Republic of China}
\email{caili@cnu.edu.cn}

\author{Yangyu Fan}
\address{Academy for Multidisciplinary Studies\\
Beijing National Center for Applied Mathematics\\
Capital Normal University\\
Beijing, 10048, People's Republic of China}
\email{b452@cnu.edu.cn; fanyangyu@amss.ac.cn}

\maketitle

\begin{abstract}
 In this short note, we compute higher extension groups for all irreducible representations and deduce the multiplicity formula for finite length representations  in triple product case. 

\end{abstract}

\tableofcontents

\section{Introduction}

Let $F$ be a $p$-adic field. 
Let $L/F$ be a cubic \'etale extension and $D/F$ be a quaternion algebra. Let $G= \Res_{L/F}D^\times$ and $H= D^\times$.
Note that the intersection of the center $Z_G$ of $G$ with 
$H$ is $Z_H=\BG_m$. Denote by $\Rep(F^\times \bs G(F))$ the category of smooth $F^\times \bs G(F)$-representations.

In this short note, we prove that  the higher Ext groups  vanish
\[\Ext^i_{F^\times \bs H(F)}(\pi,\BC) = 0, \quad i > 0.\]
 for any  generic $\pi \in \Rep(F^\times \bs G(F))$.
 Combining with results from
the local trace formula approach, we obtain 
a multiplicity formula for any irreducible $\pi \in \Rep(F^\times \bs G(F))$. 

To state the result,
we  introduce more notations. For any irreducible $\pi \in \Rep(F^\times \bs G(F))$,  consider its geometric multiplicity
\[m_{\mathrm{geo}}(\pi) = \sum_{T \in \CT} |W(H,T)|^{-1} \int_{F^\times\bs T(F)} c_\pi(t) D^H(t) dt\]
where
\begin{itemize}
	\item the support $\CT = \CT(G,H)$ 
		is a set of tori in $ H$. If $D$ is split, $\CT$ consists of $F^\times$ and the nonsplit maximal tori in $H$. 
		If $D$ is non-split, $\CT$ consists of the nonsplit maximal tori in $H$. 
	\item $W(H,T) = N_H(T)/ Z_H(T)$.
	\item $c_\pi$ is the regularized character on the semi-simple locus of $G(F)$ 
		(See \cite[Definition 2.5]{Wan21}).
	\item $D^H(t) = |\det(1-\ad(t))|_{\fh/\fh_t}|$ is the Weyl discriminant on the semi-simple
		locus of $H(F)$.
	\item the measure $dt$ is normalized so that the volume $\Vol(F^\times\bs T(F),dt) = 1$.
\end{itemize}

Via the local trace formula approach, Wan proves the following multiplicity formula
for tempered representations: 
\begin{thm}[Wan \cite{Wan17}, C.2 and C.3] \label{wan}
For any irreducible tempered $\pi \in \Rep(F^\times \bs G(F))$,
\[m(\pi): =\dim_\BC\Hom_{H(F)}(\pi,\BC) = m_{\mathrm{geo}}(\pi).\]
\end{thm}

\begin{remark}
    In fact, the above multiplicity formula holds for generic representations combining  the result of Prasad \cite{Pra90, Pra92}
    (See \cite[Remark C. 2.3]{Wan17}).
\end{remark}

In general, it is conjectured in \cite[Conjecture 7.1]{Pra18} (for the Gan-Gross-Prasad models) 
and \cite[Conjecture 7.6]{Wan21} (for all spherical pairs) that
the multiplicity formula should hold for all   $\pi\in\Rep(F^\times\bs G(F))$ of finite length
with the multiplicity $m(\pi)$ replaced by the Euler-Poincar\'e number
\[\EP_{F^\times \bs H(F)}(\pi,\BC):=\sum_{i\geq0}(-1)^i\dim_\BC\Ext^i_{F^\times \bs H(F)}(\pi,\BC).\]
Philosophically, such multiplicity formula can be viewed as a kind of Riemann-Roch theorem (See \cite[Remark 7.2]{Pra18}).

\begin{example}
Assume $D$ is split and $L = E \oplus F$ where $E/F$ is a quadratic field extension. Let $\pi = \pi_1 \boxtimes \pi_2 \in \Rep(F^\times \bs G(F))$ with
\begin{itemize}
    \item $\pi_1=I_{B(E)}^{\GL_2(E)} \chi_1\boxtimes\chi_2\in\Rep(F^\times\bs\GL_2(E))$ being the normalized parabolic induction for characters $\chi_1,\chi_2:\ E^\times \to \BC^\times$ such that $\chi_1\chi_2|_{F^\times} = 1$;
    \item $\pi_2=\BC\in\Rep(F^\times\bs \GL_2(F))$ being the trivial representation.
\end{itemize} 
Then 
\[\Hom_{F^\times\bs H(F)}(\pi,\BC) = \Hom_{F^\times\bs \GL_2(F)}(\pi_1,\BC).\]
Set $\chi' := \chi_1\ov{\chi_2}$ where $-$ denotes the Galois conjugation with respect to
$E/F$. It is known that (see e.g. \cite[Theorem 5.2]{Mat11}) $m(\pi) \leq 1$ and the equality holds if and only if
$\chi_1|_{F^\times} = \chi_2|_{F^\times} = 1$ or $\chi' = 1$. 

On the other hand, by  a property of the regularized characters of  parabolic inductions \cite[Proposition 2.7]{Wan21}
\[m_\mathrm{geo}(\pi) = \frac{1}{2} \int_{F^\times \bs E^\times}  c_\pi(t) D^H(t) dt = 
\frac{1}{2} \int_{F^\times \bs E^\times}  (\chi'(t) + \chi'(\bar{t}))dt.\]
In particular, $m_\mathrm{geo}(\pi) \leq 1$ and the equality holds if and only if $\chi' = 1$. 

Therefore, in the case $\chi' \not= 1$ and $\chi_1|_{F^\times} = \chi_2|_{F^\times} = 1$, 
\[m(\pi) \not= m_\mathrm{geo}(\pi).\]
This is compatible with  Theorem \ref{wan} since $\pi_2$ (hence $\pi$) is neither generic nor tempered.
\end{example}

The following is the main result of this paper.
\begin{thm}\label{main}
   For any irreducible $\pi\in\Rep(F^\times \bs G(F))$,  generic if $D$ is split,  
    \[\Ext^i_{F^\times \bs H(F)}(\pi,\BC) = 0, \quad i\geq 1.\]
    Moreover, for any $\pi \in \Rep(F^\times \bs G(F))$ of finite
    length,  the multiplicity formula holds
   \[\mathrm{EP}_{F^\times \bs H(F)}(\pi,\BC) = m_\mathrm{geo}(\pi).\]
\end{thm}

\begin{remark}  Previously,
\begin{itemize}
\item The vanishing result is known for the Whittaker model: if $G$ is a  connected quasi-split reductive
       group over $F$, $B = TN$ a Borel subgroup of $G$, $\psi$ a generic character on $N(F)$, then for
       any irreducible representation $\pi\in \Rep(G(F))$ and any $i \geq 1$,
       by \cite[Proposition 2.8]{Pra18}
       \[\Ext_{N(F)}^i(\pi,\psi^{-1}) \cong \Ext_{G(F)}^i\left(i_{N(F)}^{G(F)} \psi, \pi^\vee\right) \cong \Ext_1^i(\pi_{N,\psi},\BC) = 0.\]      
       For the Gan-Gan-Prasad models of general linear groups, the vanishing result is
		due to Chan-Savin \cite{CS}.
\item  The only known cases of \cite[Conjecture 7.6]{Wan21}
	are the group case \cite[Proposition 2.1(4)]{Pra18}, 
	the Whittaker models \cite[Section 8.1]{Wan21}, 
	and the Gan-Gross-Prasad models for general linear groups (\cite[Theorem 4.2]{Pra18}
	and note that in this case the support $\CT = \CT(G,H)$ of $m_\mathrm{geo}$ is $\{1\}$).

\end{itemize}
\end{remark}

We explain the proof of Theorem \ref{main}. 
In fact, it is known that $\Ext^i_{F^\times\bs H(F)}(\pi,\BC) = 0$
for any $\pi$ when $i \geq 2$ (see \cite[Proposition 2.9]{Pra18},  also Proposition  \ref{Ext}(1) below) and $\Ext^1_{F^\times\bs H(F)}(\pi,\BC) = 0$ if $\pi$ is supercuspidal (see \cite[Theorem 5.4.1]{Cas95},  also
Proposition \ref{Ext}(1) below).  The proof of the vanishing result
is reduced to showing that $\Ext^1_{F^\times\bs H(F)}(\pi,\BC) = 0$
for non-supercuspidal $\pi$. For this,
we apply standard tools: the geometric lemma of Bernstein-Zelevinsky
and the Schneider-Stuhler duality
(\cite[Theorem 2]{NP20}, see also Theorem \ref{SSNP} below). 
Once  the vanishing results are available, 
the multiplicity formula for finite length representations 
is deduced from the tempered version (Theorem \ref{wan}) 
by noting both sides are additive and constant in an unramified twisting family.

In fact, we compute  $\Ext^i_{F^\times \bs H(F)}(\pi,\BC)$ 
for all irreducible $\pi\in\Rep(F^\times \bs G(F))$  and all $i$. 
The results for $L=E\oplus F$ together with the Schneider-Stuhler duality  implies  the following complete classification of irreducible
$\GL_2(F)$-subrepresentations  of irreducible $\GL_2(E)$-representations, 
which is analogous to \cite[Proposition 9.1]{NP20} for the case $L=F\oplus F \oplus F$ (in fact, there is an extra central character condition in \cite[Proposition 9.1]{NP20}:
$\pi = \pi_1 \boxtimes \pi_2 \boxtimes \pi_3$ with $\pi_3$ having trivial central character. This condition is dropped here).

\begin{prop}\label{Strongform}Assume $L=E\oplus F$ and $D$ is split. 
For $\pi=\pi_1\boxtimes\pi_2\in\Rep(F^\times\bs G(F))$  irreducible,  
$\pi_2^\vee$ is a $\GL_2(F)$-subrepresentation of $\pi_1$ if and only if 
\begin{itemize}
\item  $\pi_1=\xi\circ\det$ and $\pi_2=\xi^{-1}|_{F^\times}\circ\det$; or
    \item $\pi_2^\vee$ is supercuspidal, and appears as a quotient of $\pi_1|_{\GL_2(F)}$; or 
    \item $\pi_2=\St_F\otimes\xi$, $\pi_1=I_{B(E)}^{\GL_2(E)}\chi_1\boxtimes\chi_2$ 
such that $\xi\chi_1|_{F^\times}=\xi\chi_2|_{F^\times}=1$ and $\xi_E\chi^\prime\neq1$. 
Here $\xi_E=\xi\circ N_{E/F}$,  $\chi' = \chi_1\ov{\chi_2}$ and $\St_F$ is the Steinberg representation for $\GL_2(F)$.   
\end{itemize}
 \end{prop}

\s{\bf Acknowledgement} The debt this work owes to \cite{Pra18} and \cite{NP20} is clear. We thank Professor Ye Tian for his consistent encouragement. We want to thank Professor Kei-Yuen Chan for helpful discussions. 


\section{The proof}
 For any reductive group $\Gamma$ over $F$ and any center character $\omega:\ Z_\Gamma(F)\to\BC^\times$, let $\Rep(\Gamma(F),\omega)$ denote the full subcategory of $\Rep(\Gamma(F))$ consisting of objects on which $Z_\Gamma(F)$ acts by $\omega$. For any objects $\pi,\pi^\prime\in\Rep(\Gamma(F),\omega)\subset \Rep(\Gamma(F))$, let $\Ext^i_{\Rep(\Gamma(F),\omega)}(\pi,\pi^\prime)$ (resp. $\Ext^i_{\Gamma(F)}(\pi,\pi^\prime)$) be the $i$-th extension group in the category $\Rep(\Gamma(F),\omega)$ (resp. $\Rep(\Gamma(F))$). 
 Let  $\Gamma^\prime(F)$ be  a closed subgroup of the $p$-adic group $\Gamma(F)$ with  modulus character $\delta_{\Gamma^\prime}$. Note that   when  $\Gamma^\prime$ is reductive, $\delta_{\Gamma^\prime}$ is trivial. Denote by $I_{\Gamma^\prime(F)}^{\Gamma(F)}$ (resp. $i_{\Gamma^\prime(F)}^{\Gamma(F)}$) the normalized (resp. compact) induction.  Note that for any $\sigma\in \Rep(\Gamma^\prime(F))$, $(i_{\Gamma^\prime(F)}^{\Gamma(F)}\sigma)^\vee=I_{\Gamma^\prime(F)}^{\Gamma(F)}\sigma^\vee$
where $-^\vee$ stands for smooth dual in the proper category. Moreover if $\Gamma^\prime=MN$ is a parabolic subgroup with Levi factor $M$,  denote by  $J_N$   the normalized Jacquet functor  for $\Gamma^\prime$.

We record some basic properties of Ext-groups which are frequently used in the following.
 \begin{prop}\label{Ext} 
 We have the following results for the Ext-groups.
 \begin{enumerate}
 \item For any irreducible $\pi\in \Rep(\Gamma(F))$,  smooth $\pi^\prime\in \Rep(\Gamma(F))$ and any $i > \mathrm{split\ rank\ of}\ Z_\Gamma$
  $$\Ext^i_{\Gamma(F)}(\pi,\pi^\prime) \cong \Ext^i_{\Gamma(F)}(\pi^\prime,\pi) = 0.$$
  If $\pi$ is supercuspidal with central character $\omega$, then $\pi$  is both projective and injective in $\Rep(\Gamma(F),\omega)$. In particular,
  for any $i \geq 1$
  $$\Ext^i_{\Rep(\Gamma(F),\omega)}(\pi,\pi^\prime) \cong \Ext^i_{\Rep(\Gamma(F),\omega)}(\pi^\prime,\pi) = 0.$$
 \item For any $\pi\in \Rep(\Gamma(F),\omega)$, $\sigma\in \Rep(\Gamma(F),\omega^{-1})$ and any $i \geq 0$
  $$\Ext^i_{Z_\Gamma(F)\bs \Gamma(F)}(\pi\otimes\sigma,\BC) \cong \Ext^i_{\Rep(\Gamma(F),\omega)}(\pi,\sigma^\vee)\cong \Ext^i_{\Rep(\Gamma(F),\omega^{-1})}(\sigma,\pi^\vee).$$
 \item  The restriction functor 
 $\Rep(\Gamma(F))\to \Rep(\Gamma^\prime(F))$
 sends projective objects to projective objects, the normalized induction functor $I_{\Gamma^\prime(F)}^{\Gamma(F)}:\ \Rep(\Gamma^\prime(F))\to \Rep(\Gamma(F))$ sends injective objects to injective objects.  Moreover, for any $\sigma\in \Rep(\Gamma^\prime(F))$, $\pi\in \Rep(\Gamma(F))$
 and any $i \geq 0$
 $$\Ext^i_{\Gamma(F)}\left(\pi, I_{\Gamma^\prime(F)}^{\Gamma(F)}\sigma\right) \cong \Ext^i_{\Gamma^\prime(F)}\left(\pi, \sigma\otimes\delta_{\Gamma^\prime}^{1/2}\right).$$ 
 When $\Gamma^\prime=MN$ is a parabolic subgroup with Levi factor $M$,  the normalized Jacquet functor $J_N:\ \Rep(\Gamma(F))\to \Rep(M(F))$ sends projective objects to projective objects. For any $\sigma\in \Rep(M(F))$ and any $i \geq 0$
  $$\Ext^i_{\Gamma(F)}\left(\pi,I_{\Gamma^\prime(F)}^{\Gamma(F)}\sigma\right) \cong  \Ext^i_{M(F)}(J_N(\pi),\sigma).$$
 \end{enumerate}
 \end{prop}
\begin{proof}The property of supercuspidal representations can be found in \cite[Theorem 5.4.1]{Cas95}. Other results are summarized in   \cite[Section 2]{Pra18}.
    \end{proof}

For any $\pi\in\Rep(\Gamma(F),\omega)$ irreducible, let $d(\pi)$ (resp. $d^\prime(\pi)$) be  the split rank of $Z_M$ (resp. $Z_M\cap[\Gamma,\Gamma]$), where $M$ is any Levi subgroup carrying the cuspidal support of $\pi$. The key ingredient for computing higher Ext groups is the following Schneider-Stuhler duality theorem.
 \begin{thm}[Theorem 1,2 in \cite{NP20}] \label{SSNP} 
 Let $\pi\in\Rep(\Gamma(F),\omega)$ be any irreducible representation and $D(\pi)$ be the Aubert-Zelevinsky involution of $\pi$.  Then 
 \begin{itemize}
     \item for any $\pi^\prime\in\Rep(\Gamma(F))$, $\Ext^i_{\Gamma(F)}(\pi,\pi^\prime) = 0$ for $i > d(\pi)$ and for $0 \leq i \leq d(\pi)$, there is a non-degenerate pairing 
 $$\Ext^i_{\Gamma(F)}(\pi,\pi^\prime) \times \Ext^{d(\pi)-i}_{\Gamma(F)}(\pi^\prime, D(\pi))\to \Ext_{\Gamma(F)}^{d(\pi)}(\pi,D(\pi)) \cong \BC,$$
 \item for any $\pi^\prime\in\Rep(\Gamma(F),\omega)$, $\Ext^i_{\Rep(\Gamma(F),\omega)}(\pi,\pi^\prime) = 0$ for $i > d^\prime(\pi)$ and for $0 \leq i \leq d^\prime(\pi)$, there is a non-degenerate pairing 
 $$\Ext^i_{\Rep(\Gamma(F),\omega)}(\pi,\pi^\prime)\times \Ext^{d^\prime(\pi)-i}_{\Rep(\Gamma(F),\omega)}(\pi^\prime, D(\pi))\to \Ext_{\Rep(\Gamma(F),\omega)}^{d^\prime(\pi)}(\pi,D(\pi)) \cong \BC .$$
 \end{itemize}
  \end{thm}
  \begin{remark}In general, the non-degeneracy means that if $\pi^\prime=\varinjlim_n \pi_n^\prime$ is the inductive limit of  finite generated $\Gamma(F)$-submodules $\pi_n^\prime$, then 
 $$\Ext^i_{\Gamma(F)}(\pi,\pi^\prime) = \varinjlim_n  \Ext^i_{\Gamma(F)}(\pi,\pi_n^\prime),\quad \Ext^i_{\Gamma(F)}(\pi^\prime, D(\pi)) = \varprojlim_n  \Ext^i_{\Gamma(F)}(\pi_n^\prime, D(\pi))$$
 is the inductive limit (resp. projective limit) of finite dimensional $\BC$-spaces and  the pairing is direct limit of perfect pairings on these finite dimensional spaces. 
 
 In the triple product case,  the result of Aizenbud-Sayag \cite{AS20} guarantees all the Ext-groups below are finite dimensional and  the non-degeneracy has the usual meaning.
  \end{remark}

In the following, we shall concentrate on the triple product case. 

The case $D$ non-split is straightforward.
\begin{prop}Assume $D$ is non-split. Then for any $\pi\in\Rep(F^\times\bs G(F))$ irreducible, 
\[\Ext^i_{F^\times \bs H(F)}(\pi,\BC) = 0, \quad i > 0.\]
\end{prop}
\begin{proof}Note that $Z_H\bs H$ is anisotropic. Then by Proposition \ref{Ext} (1), $\BC$  is injective and the statement follows. 
\end{proof}
To deal with  the $D$ split case, we need to consider 
the following Waldspurger toric case.
\begin{lem}\label{EP-Wal}Assume $D$ is split. Let $K\subset D$ be an \'etale quadratic $F$-algebra and embed $F^\times$ into $D^\times(F)\times K^\times$ diagonally. Then for any generic irreducible $\pi\in\Rep(F^\times\bs (D^\times(F)\times K^\times)\big)$,  $$\Ext^i_{F^\times\bs K^\times}(\pi,\BC)=0,\quad i\geq1.$$
\end{lem}
\begin{proof}
Assume $\pi=\sigma\boxtimes\xi$ with $\sigma\in\Rep(D^\times(F))$ and $\xi\in \Rep(K^\times)$. If $K/F$ is a field extension, then   
$\Ext^i_{F^\times\bs K^\times}(\sigma\boxtimes\xi,\BC)=0$ for $i\geq1$ by Proposition \ref{Ext} (1).

If $K=F\oplus F$, $\Ext^i_{F^\times\bs K^\times}(\sigma\boxtimes\xi,\BC)=0$ for $i\geq2$ by Proposition \ref{Ext} (1). 
For $i=1$, one can apply Theorem \ref{SSNP} to $\BC\in \Rep(F^\times\bs K^\times)$, where $d(\BC)=1$ and $D(\BC)=\BC$. One has $$\dim_\BC \Ext^1_{F^\times\bs K^\times}(\sigma\boxtimes\xi, \BC)=\dim_\BC \Hom_{F^\times\bs K^\times}(\BC, \sigma\boxtimes\xi)=\dim_\BC \Hom_{K^\times}(\xi^{-1}, \sigma).$$
     Identify $K^\times$ with the split torus $T(F)$ and  let $\CK(\sigma)$ be the Kirillov model of $\sigma$. Write $\xi=\xi_1\boxtimes\xi_2$. Then by \cite[Proposition 4.7.2]{Bum97}, $$\Hom_{T(F)}(\xi^{-1},\CK(\sigma)|_T)=\{\phi\in \CK(\sigma)\ |\  \phi(ax)=\xi_1^{-1}(a)\phi(x) \text{ for any } a\in F^\times\}=0.$$
\end{proof}


Assume $D$ is split. By Proposition \ref{Ext}(1), $\Ext^i_{F^\times \bs H(F)}(\pi,\BC) = 0$ for $i \geq 2$. In the following, we compute $\Ext^1_{F^\times \bs H(F)}(\pi,\BC)$ case by case.   Let $T\subset B\subset \GL_2$ be the diagonal torus and the group of upper triangular matrices respectively,  $\St_F$ be the Steinberg representation for $\GL_2(F)$ and $|\cdot|_F$ be the norm character on $F^\times$.
\begin{prop}\label{case I}Assume $D$ is split and $L/F$ is a cubic field extension. Then for  $\pi\in\Rep(F^\times\bs G(F))$ irreducible, $\Ext^1_{F^\times\bs H(F)}(\pi,\BC)=0$.
 \end{prop}
\begin{proof}By Proposition \ref{Ext} (1,3)(the modulus character of $H$ is trivial), when $\pi$ is supercuspidal, $$\Ext^1_{F^\times\bs H(F)}(\pi,\BC) = \Ext^1_{F^\times\bs G(F)}(\pi,I_{H(F)}^{G(F)}\BC) = 0.$$

If $\pi$ is one-dimensional, then by Theorem \ref{SSNP} for $\BC\in\Rep(F^\times\bs H(F))$, where $D(\BC)=\St_F$ and $d(\BC)=\BC$,
$$\dim_\BC\Ext^1_{F^\times\bs H(F)}(\pi,\BC) = \dim_\BC\Hom_{H(F)}(\St_F, \pi) = 0.$$

Note that when $\pi$ is a special series, then $\pi\hookrightarrow I_{B(L)}^{\GL_2(L)}\chi_1\boxtimes\chi_2$  with $\chi_1\chi_2^{-1}=|\cdot|_L$. Thus it suffices to show that for $\pi= I_{B(L)}^{\GL_2(L)}\chi\in\Rep(F^\times\bs G(F))$ where $\chi=\chi_1\boxtimes\chi_2$ with $\chi_1\chi_2^{-1}\neq|\cdot|_L^{-1}$, 
$$\Ext^1_{F^\times\bs H(F)}(\pi,\BC)=0.$$
By the geometric lemma (see \cite[Section 6.1]{Pra92}), there exists an exact sequence of $F^\times\bs H(F)$-representations
$$0\to i_{F^\times}^{H(F)}\BC\to \pi \to I_{B(F)}^{H(F)}\chi\delta_B\to0$$
where  $\delta_B$ is the modulus character of $B(F)$. The short exact sequence induces a long exact sequence
\begin{align*}
  0 \to &\Hom_{F^\times\bs H(F)}(I_{B(F)}^{H(F)}\chi\delta_B,\BC) \to \Hom_{F^\times\bs H(F)}(\pi,\BC) \to \Hom_{F^\times\bs H(F)}(i_{F^\times}^{H(F)}\BC,\BC) \to\\
&  \Ext^1_{F^\times\bs H(F)}(I_{B(F)}^{H(F)}\chi\delta_B,\BC) \to \Ext^1_{F^\times\bs H(F)}(\pi,\BC) \to \Ext^1_{F^\times\bs H(F)}(i_{F^\times}^{H(F)}\BC,\BC) \to 0.
\end{align*}
By Proposition \ref{Ext}(2,3),
$$\Hom_{F^\times\bs H(F)}(i_{F^\times}^{H(F)}\BC,\BC)=\BC,\quad \Ext^1_{F^\times\bs H(F)}(i_{F^\times}^{H(F)}\BC,\BC)=0.$$
Note that 
\begin{itemize}
    \item if $\pi$ is reducible, $I_{B(F)}^{H(F)}\chi\delta_B$ is irreducible;
    \item if $\pi$ is irreducible,   we may assume $I_{B(F)}^{H(F)}\chi\delta_B$ is irreducible by replacing  $\chi$ by $\chi^w$, the twisting of $\chi$ by $w=\begin{pmatrix} 0 & 1\\ 1 & 0\end{pmatrix}$, if necessary.
\end{itemize} 
Thus   $$\Hom_{F^\times\bs H(F)}(I_{B(F)}^{H(F)}\chi\delta_B,\BC)=0.$$
Moreover by Theorem \ref{SSNP} for $\BC\in\Rep(F^\times\bs H(F))$, where $D(\BC)=\St_F$ and $d(\BC)=1$,
$$\dim_\BC\Ext^1_{F^\times\bs H(F)}(I_{B(F)}^{H(F)}\chi\delta_B,\BC) = \dim_\BC\Hom_{F^\times\bs H(F)}(\St_F, I_{B(F)}^{H(F)}\chi\delta_B) = 0.$$
Consequently, $\Hom_{H(F)}(\pi,\BC) = \BC$ and $\Ext^1_{F^\times\bs H(F)}(\pi,\BC) = 0$.
\end{proof} 
\begin{prop}\label{case II}Assume $D$ is split and  $L=E\oplus F$. For $\pi=\pi_1\boxtimes\pi_2\in\Rep(F^\times\bs G(F))$ irreducible, $\dim_\BC\Ext^1_{F^\times\bs H(F)}(\pi,\BC) \leq 1$ with the equality holds if and only if \begin{itemize}
        \item $\pi_2=\xi\circ\det$, $\pi_1=I_{B(E)}^{\GL_2(E)}\chi_1\boxtimes\chi_2$ such that $\chi_1|_{F^\times}\xi=\chi_2|_{F^\times}\xi=1$ and $\xi_E\chi^\prime\neq1$
        (See Proposition \ref{Strongform} for the notations $\xi_E$ and $\chi'$); or
        \item $\pi_1=\xi\circ\det$ and $\pi_2=\St_F\otimes\xi^{-1}|_{F^\times}$. 
    \end{itemize}
\end{prop}
\begin{proof}  Denote by $\omega^{-1}$ the central character of $\pi_2$. Then by  Proposition \ref{Ext} (2) 
$$\Ext^i_{F^\times\bs H(F)}(\pi,\BC) \cong \Ext^i_{\Rep(H(F),\omega)}(\pi_1,\pi_2^\vee)$$
and   if $\pi_1$ or $\pi_2$ is supercuspidal, $\Ext^1_{F^\times\bs H(F)}(\pi,\BC)=0$.

When $\pi_1$ is one-dimensional,  by Theorem \ref{SSNP} for $\BC\in\Rep(F^\times\bs H(F))$,  where $D(\BC)=\St_F$ and $d(\BC)=1$,
$$\dim_\BC\Ext^1_{F^\times\bs H(F)}(\pi,\BC)=\dim_\BC\Hom_{H(F)}(\St_F, \pi)\leq 1$$
with the equality holds if and only if $\pi_1=\xi\circ\det$ and $\pi_2=\St_F\otimes\xi^{-1}|_{F^\times}$. 

Note that when $\pi_1$ is a special series,  $\pi\hookrightarrow I_{B(E)}^{\GL_2(E)}\chi_1\boxtimes\chi_2$ with $\chi_1\chi_2^{-1}=|\cdot|_E$. Thus it suffices to consider $\pi_1= I_{B(E)}^{\GL_2(E)}\chi$ for $\chi=\chi_1\boxtimes\chi_2$ where $\chi_1\chi_2^{-1}\neq|\cdot|_E^{-1}$. By the geometric lemma (See \cite[Section 4.1]{Pra92}), there is an  exact sequence of $H(F)$-representations
$$0\to i_{E^\times}^{H(F)}\chi^\prime\to \pi_1
\to I_{B(F)}^{H(F)}\chi\delta_B^{1/2}\to0.$$
By Proposition \ref{Ext}(2,3) and Lemma \ref{EP-Wal},
$$ \Hom_{H(F)}(i_{E^\times}^{\GL_2(F)}\chi^\prime,\pi_2^\vee)\cong  \Hom_{E^\times}(\pi_2\otimes\chi^\prime, \BC);$$
$$ \Ext^1_{\Rep(H(F),\omega)}(i_{E^\times}^{\GL_2(F)}\chi^\prime,\pi_2^\vee)\cong  \Ext^1_{F^\times\bs E^\times}(\pi_2\otimes\chi^\prime, \BC)=0;$$
$$\Ext^i_{\Rep(H(F),\omega)}(I_{B(F)}^{\GL_2(F)}\chi\delta_B^{1/2},\pi_2^\vee)\cong\Ext^i_{\Rep(T(F),\omega^{-1})}(J_N(\pi_2),\delta_B^{-1/2}\chi^{-1}) \cong \Ext^i_{F^\times\bs T(F)}(J_N(\pi_2)\otimes\delta_B^{1/2}\chi,\BC).$$
Then by Theorem \ref{SSNP} for $\BC\in \Rep(F^\times\bs T(F))$, where $D(\BC)=\BC$ and $d(\BC)=1$,
$$\dim_\BC\Ext^1_{\Rep(H(F),\omega)}(I_{B(F)}^{\GL_2(F)}\chi\delta_B^{1/2},\pi_2^\vee)=\dim_\BC\Hom_{H(F)}(I_{B(F)}^{\GL_2(F)}\chi\delta_B^{1/2},\pi_2^\vee)\leq 1$$
 with the equality holds if and only if $\delta_B^{-1/2}\chi^{-1}|_{T(F)}$ is a Jordan factor of $J_N(\pi_2)$.
By the long exact sequence
\begin{align*}
  0\to &\Hom_{H(F)}(I_{B(F)}^{H(F)}\chi\delta_B^{1/2},\pi_2^\vee)\to \Hom_{H(F)}(\pi_1,\pi_2^\vee)\to
  \Hom_{H(F)}(i_{E^\times}^{H(F)}\chi^\prime,\pi_2^\vee)\to \\
 & \Ext^1_{\Rep(H(F),\omega)}(I_{B(F)}^{H(F)}\chi\delta_B^{1/2},\pi_2^\vee)\to 
   \Ext^1_{\Rep(H(F),\omega)}(\pi_1,\pi_2^\vee)\to \Ext^1_{\Rep(H(F),\omega)}(i_{E^\times}^{H(F)}\chi^\prime,\pi_2^\vee)\to0,
 \end{align*}
one has that
\begin{enumerate}[(i)]
    \item $\EP_{F^\times\bs H(F)}(\pi,\BC) = \dim_\BC \Hom_{E^\times}(\pi_2\otimes\chi^\prime, \BC) \leq 1$
\item unless $\delta_B^{-1/2}\chi^{-1}|_{T(F)}$ is a Jordan factor of $J_N(\pi_2)$, $\Ext^1_{F^\times\bs H(F)}(\pi,\BC)=0,$
    \item if $\Hom_{E^\times}(\pi_2\otimes\chi^\prime, \BC)=0$, 
    $$\dim_\BC\Ext^1_{F^\times\bs H(F)}(\pi,\BC)=\dim_\BC\Ext^1_{\Rep(H(F),\omega)}(I_{B(F)}^{H(F)}\chi\delta_B^{1/2},\pi_2^\vee).$$
\end{enumerate}

If $\delta_B^{-1/2}\chi^{-1}|_{T(F)}$ is a Jordan factor of $J_N(\pi_2)$ and $\pi_1=I_{B(E)}^{\GL_2(E)}\chi$ sits in the exact sequence
$$0\to\St_E\otimes\mu\to \pi_1\to \mu\to0,$$
  one has  $\pi_2=I_{B(F)}^{\GL_2(F)}\chi^{-1}\delta_B^{-1/2}$  is irreducible. 
Then by Saito-Tunnell, $\EP_{F^\times\bs H(F)}(\pi,\BC)=1$ and hence $$\dim_\BC \Hom_{H(F)}(\St_E\otimes\mu,\pi_2^\vee)=\dim_\BC \Ext^1_{\Rep(H(F),\omega)}(\St_E\otimes\mu,\pi_2^\vee)+1.$$
Note that $$\dim_\BC\Hom_{H(F)}(\St_E\otimes\mu,\pi_2^\vee)=\dim_\BC\Hom_{F^\times\bs H(F)}(\St_E\otimes\mu\otimes \pi_2,\BC)\leq1.$$
one deduce $\Ext^1_{F^\times\bs H(F)}(\St_E\otimes\mu\otimes \pi_2,\BC)=0$ and hence $\Ext^1_{F^\times\bs H(F)}(\pi,\BC)=0$  from the induced long exact sequence.

 Assume $\pi_1=I_{B(E)}^{\GL_2(E)}\chi$ is irreducible.  If  $\chi_1|_{F^\times}\neq \chi_2|_{F^\times}$, then  up to replacing $\chi$ by $\chi^w$, one can make $\delta_B^{-1/2}\chi^{-1}|_{T(F)}$ different from the Jordan Holder factors of $J_N(\pi_2)$. Then 
$\Ext^1_{F^\times\bs H(F)}(\pi,\BC)=0$ by (ii). 

Now assume moreover  $\xi^{-1}=\chi_1|_{F^\times}= \chi_2|_{F^\times}$. Then $\delta_B^{-1/2}\chi^{-1}|_{T(F)}$ is a Jordan factor of $J_N(\pi_2)$ if and only if $\pi_2=\xi\circ\det$ or $\St_F\otimes\xi$. In the case $\pi_2=\St_F\otimes\xi$, 
 \begin{itemize}
     \item if $\EP_{F^\times\bs H(F)}(\pi,\BC) = 1$, then  $\Ext^1_{F^\times\bs H(F)}(\pi,\BC)=0$ since $\dim_\BC\Hom_{H(F)}(\pi,\BC) \leq 1$;
     \item if $\EP_{F^\times\bs H(F)}(\pi,\BC) = \dim_\BC\Hom_{E^\times}(\pi_2\otimes\chi^\prime, \BC) = 0$, then by Theorem \ref{SSNP} for $\St_F\in\Rep(F^\times\bs H(F))$, where $D(\St_F)=\BC$ and $d(\St_F)=1$,
     \begin{align*}
         \dim_\BC\Ext^1_{F^\times\bs H(F)}(\pi,\BC) &= \dim_\BC\Ext^1_{F^\times\bs H(F)}(I_{B(F)}^{H(F)}\chi\delta_B^{1/2}\otimes \xi,\St_F) \\
                                                    &= \dim_\BC \Hom_{F^\times\bs H(F)}(\BC,  I_{B(F)}^{\GL_2(F)}\delta_B^{1/2})=0.
     \end{align*}
 \end{itemize}

In the case  $\pi_2=\xi\circ\det$,
we shall give the criterion for $\Ext_{F^\times \bs H(F)}^1(\pi,\BC) = \BC$ via studying $\EP_{F^\times\bs H(F)}(\pi,\BC)$.
By the above (i), $\EP_{F^\times\bs H(F)}(\pi,\BC)= \dim_\BC \Hom_{E^\times}(\pi_2\otimes\chi^\prime, \BC) \leq 1$,   where the equality holds if and only if $\xi_E\chi^\prime=1$.  

If $\EP_{F^\times\bs H(F)}(\pi,\BC) = 1$, as the second term in the long exact sequence has multiplicity $\leq1$, we must have
$\Ext_{F^\times \bs H(F)}^1(\pi,\BC) = 0$. 

If $\EP_{F^\times\bs H(F)}(\pi,\BC) = 0$, then as the last term in the last exact sequence vanishes, we must have
\[\dim_\BC\Hom_{E^\times}(\pi_2\otimes\chi^\prime,\BC) = 0, \quad 
 \dim_\BC \Ext_{F^\times \bs H(F)}^1(\pi,\BC) = \dim_\BC \Hom_{H(F)}(I_{B(F)}^{H(F)} \chi\delta_B^{1/2},\pi_2^\vee).\]
 It is then easy to see that $\dim_\BC \Ext_{F^\times \bs H(F)}^1(\pi,\BC) = 1$ if and only if $\chi|_{F^\times} \xi = 1$
 and $\xi_E \chi' \not= 1$.

\end{proof}
 One can deduce Proposition \ref{Strongform} immediately from Proposition \ref{case II}. 
\begin{proof}[Proof of Proposition \ref{Strongform}]Denote the central character of $\pi_2$ by $\omega^{-1}$. By Theorem 
	\ref{SSNP},
\begin{itemize}
    \item when $\pi_2$ is supercuspidal, by Proposition \ref{Ext} (1),  $$\Hom_{H(F)}(\pi_2^\vee,\pi_1) \cong \Hom_{H(F)}(\pi_1,\pi^\vee_2);$$
    \item when $\pi_2$ is non-supercuspidal, $d(\pi_2)=1$ and 
    $$\dim_\BC \Hom_{H(F)}(\pi^\vee_2,\pi_1) = \dim_\BC \Ext^1_{\Rep(H(F),\omega)}(\pi_1,D(\pi^\vee_2)) = \dim_\BC \Ext^1_{F^\times\bs H(F)}(\pi_1\otimes D(\pi_2),\BC).$$
  \end{itemize}
Since $D(\xi\circ\det)=\St_F\otimes\xi$ and $D(\St_F\otimes\xi)=\xi\circ\det$, the statement follows from  Proposition \ref{case II} immediately.
   \end{proof}
 \begin{prop}\label{case III}Assume $D$ is split and  $L=F\oplus F \oplus F$. For $\pi=\pi_1\boxtimes\pi_2\boxtimes\pi_3\in\Rep(F^\times\bs G(F))$ with  $\pi_i\in\Rep(\GL_2(F))$ irreducible, $\dim_\BC\Ext^1_{F^\times\bs H(F)}(\pi,\BC)\leq 1$  with the equality holds if and only if up to reordering and twisting,
 \begin{itemize}
     \item $\pi_1=\St_F$ and $\pi_2=\pi_3=\BC$, or
     \item $\pi_1=\pi_2^\vee$ are principal series and $\pi_3=\BC$.
 \end{itemize} 
\end{prop}
\begin{proof} Denote the central character of $\pi_3$ by $\omega^{-1}$. Then by Proposition \ref{Ext} (i) (ii)$$\Ext^1_{F^\times \bs H(F)}(\pi,\BC)\cong \Ext^1_{\Rep(H(F),\omega)}(\pi_1\otimes\pi_2,\pi_3^\vee)$$
and if  $\pi_3$ is supercuspidal,  $\Ext^1_{F^\times \bs H(F)}(\pi,\BC)=0$. In the following, we assume none of $\pi_i$ is supercuspidal.

If $\pi_3$ is one dimensional, which we may assume to be $\BC$ by twisting,  then
$$\Ext^1_{F^\times\bs H(F)}(\pi,\BC)\cong \Ext^1_{\Rep(H(F),\omega_1)}(\pi_1,\pi_2^\vee)$$
where $\omega_1$ is the central character of $\pi_1$.  Then by Theorem \ref{SSNP} for $\pi_2^\vee$, $$\dim_\BC\Ext^1_{F^\times\bs H(F)}(\pi,\BC)\leq 1$$
with the equality holds if and only if up to reordering and twisting,
\begin{itemize}
    \item $\pi_1=\St_F$ and $\pi_2=\BC$,
    \item $\pi_1=\pi_2^\vee$ are principal series.
\end{itemize}

If $\pi_1$, $\pi_2$ are special series, which we assume to be  $\St_F$ by twisting, and  $\pi_3$ is generic, then  by \cite[Lemma 5.4]{Pra90}, there exists an exact sequence of $\GL_2(F)$-representations
$$0\to\St_F\to i_{T(F)}^{\GL_2(F)}\BC\to \pi_1\otimes\pi_2\to0,$$ which induces the long exact sequence
\begin{align*}
    0\to& \Hom_{F^\times\bs H(F)}(\pi_1 \otimes \pi_2,\pi_3^\vee)\to
    \Hom_{F^\times\bs H(F)}(i_{ T(F)}^{H(F)}\BC,\pi_3^\vee)\to\Hom_{F^\times\bs H(F)}(\St_F,\pi^\vee_3)\to \\
   & \Ext^1_{F^\times\bs H(F)}(\pi_1 \otimes \pi_2,\pi_3^\vee)\to
    \Ext^1_{F^\times\bs H(F)}(i_{ T(F)}^{H(F)}\BC,\pi_3^\vee)\to
    \Ext^1_{F^\times\bs H(F))}(\St_F,\pi^\vee_3)\to0.
\end{align*}
By Lemma \ref{EP-Wal}, $$\Ext^1_{F^\times\bs H(F)}(i_{ T(F)}^{H(F)}\BC,\pi_3^\vee)=0,\quad \Hom_{F^\times \bs H(F)}(i_{ T(F)}^{H(F)}\BC,\pi_3^\vee)=\BC.$$
 Thus
\begin{itemize}
    \item when $\pi_3=\St_F$, $\Hom_{H(F)}(\pi,\BC)=0$ by \cite[Theorem 1.2]{Pra90} and hence $\Ext^1_{F^\times \bs H(F)}(\pi,\BC)=0$,
    \item when $\pi_3\neq\St_F$, $\Ext^1_{F^\times\bs H(F)}(\pi|_H,\BC)=\Hom_{F^\times\bs H(F)}(\St_F,\pi^\vee_3)=0$.
\end{itemize}

If $\pi_1=I_{B(F)}^{\GL_2(F)}\xi_1$, $\pi_2=I_{B(F)}^{\GL_2(F)}\xi_2$ and $\pi_3$ is generic. 
Then by the geometric lemma (see \cite[Section 5]{Pra90}), there is an exact sequence of $\GL_2(F)$-representations
$$0\to i_{T(F)}^{\GL_2(F)}(\xi_1\xi_2^w)\to \pi_1\otimes\pi_2\to I_{B(F)}^{\GL_2(F)}(\xi_1\xi_2\delta^{1/2})\to0,$$
which leads to  a long exact sequence
\begin{align*}
    0\to& \Hom_{H(F)}(I_{B(F)}^{\GL_2(F)}(\xi_1\xi_2\delta_B^{1/2}),\pi_3^\vee)\to\Hom_{H(F)}(\pi_1\otimes\pi_2,\pi_3^\vee)\to\\ &\Hom_{H(F)}(i_{T(F)}^{\GL_2(F)}(\xi_1\xi_2^w),\pi_3^\vee)\to
     \Ext^1_{\Rep(H(F),\omega)}(I_{B(F)}^{\GL_2(F)}(\xi_1\xi_2\delta_B^{1/2}),\pi_3^\vee)\to\\
     &\Ext^1_{\Rep(H(F),\omega)}(\pi_1\otimes\pi_2,\pi_3^\vee)\to \Ext^1_{\Rep(H(F),\omega)}(i_{T(F)}^{\GL_2(F)}(\xi_1\xi_2^w),\pi_3^\vee)\to0.
\end{align*}
By Proposition \ref{Ext} (2)(3),
$$\Ext^i_{\Rep(H(F),\omega)}(i_{T(F)}^{\GL_2(F)}(\xi_1\xi_2^w),\pi_3^\vee)=\Ext^i_{F^\times\bs T(F)}(\pi_3\otimes\xi_1\xi_2^w,\BC).$$
Moreover by Theorem \ref{SSNP} for $\pi_3^\vee\in\Rep(F^\times\bs H(F),\omega)$, where $d(\pi_3^\vee)=1$,
$$\dim_\BC\Ext^1_{\Rep(H(F),\omega)}(I_{B(F)}^{\GL_2(F)}(\xi_1\xi_2\delta_B^{1/2}),\pi_3^\vee)=\dim_\BC\Hom_{H(F)}(I_{B(F)}^{\GL_2(F)}(\xi_1\xi_2\delta_B^{1/2}),\pi_3^\vee).$$
Then by  Lemma \ref{EP-Wal}, one has
\begin{align*}\dim_\BC\Ext^1_{\Rep(H(F),\omega)}(\pi_1\otimes\pi_2,\pi_3^\vee)& = 
\dim_\BC\Hom_{F^\times\bs H(F)}(\pi,\BC)-\dim_\BC\Hom_{F^\times\bs  T(F)}(\pi_3\otimes\xi_1\xi_2^w,\BC)\\
&=\dim_\BC\Hom_{F^\times\bs H(F)}(\pi,\BC)-1\geq0.
\end{align*}
Since $\dim_\BC\Hom_{H(F)}(\pi|_H,\BC)\leq1$, one has $\Ext^1_{F^\times\bs H(F)}(\pi,\BC) = \Ext^1_{\Rep(H(F),\omega)}(\pi_1\otimes\pi_2,\pi_3^\vee) = 0.$

Interchanging the roles of $\pi_i$, the statement follows.
\end{proof}
Immediately, we deduce the following corollary,   which slightly generalizes \cite[Proposition 9.1]{NP20} (by dropping the central character condition).
\begin{cor}Let $\pi=\boxtimes_{i=1}^3\pi_i\in\Rep(F^\times\bs G(F))$ be an irreducible smooth representation. Then $\pi_3^\vee$ is a $\GL_2(F)$-subrepresentation of $\pi_1\otimes\pi_2$ if and only if 
\begin{itemize}
\item $\pi_3^\vee$ is supercuspidal, and appears as a quotient of $\pi_1\otimes\pi_2$;
\item $\pi_3^\vee=\St_F\otimes\xi$ and $\pi_1=\pi_2^\vee\otimes\xi$ is a principal series;
\item $\pi_1$ or $\pi_2$ is one-dimensional and $\pi_3^\vee\cong \pi_1\otimes\pi_2$.
\end{itemize}
\end{cor}
\begin{proof}Denote the central character of $\pi_3$ by $\omega^{-1}$. By Theorem 
	\ref{SSNP} and Proposition \ref{Ext} (2),
\begin{itemize}
    \item when $\pi_3$ is supercuspidal, $$\Hom_{H(F)}(\pi_3^\vee,\pi_1\otimes\pi_2)=\Hom_{H(F)}(\pi_1\otimes\pi_2,\pi^\vee_2),$$
    \item when $\pi_3$ is non-supercuspidal,
    $$\Hom_{H(F)}(\pi^\vee_3,\pi_1\otimes\pi_2)=\Ext^1_{F^\times\bs H(F)}(\pi_1\otimes\pi_2\otimes D(\pi_3),\BC).$$
  \end{itemize}
Then the statement follows immediately from Proposition \ref{case III}.
\end{proof}

\begin{proof}[Proof of Theorem \ref{main}]
	
The vanishing of higher Ext-groups for generic representations is proved in Propositions \ref{case I}, \ref{case II}, \ref{case III}.
In particular, for  $\pi \in \Rep(F^\times \bs G(F))$ tempered, 
\[m(\pi) = \EP_{F^\times \bs H(F)}(\pi,\BC).\]
Hence, by the multiplicity formula for tempered representations (Theorem \ref{wan}),
\[\EP_{F^\times \bs H(F)}(\pi,\BC) = m(\pi) = m_\mathrm{geo}(\pi).\]

For any Levi subgroup $M$ of $G$, denote by $\mathrm{Tmp}(F^\times \bs M(F))$ the set of tempered representations
on $F^\times \bs M(F)$ and $\wh{F^\times \bs M(F)}^{\mathrm{un}}$ the set of unramified characters on
$F^\times \bs M(F)$. Consider the subcategory $\Rep^0(F^\times \bs G(F))$ of $\Rep(F^\times \bs G(F))$
consisting of
$I_P^G (\sigma \otimes \chi)$
where $P = MN$ is a parabolic subgroup of $G$, $\sigma \in \mathrm{Tmp}(F^\times \bs M(F))$ and
$\chi \in \wh{F^\times \bs M(F)}^{\mathrm{un}}$. We have the following two lemmas.
\begin{lem}
 The Grothendieck group of the abelian sub-category of finite length representations in
		$\Rep(F^\times \bs G(F))$ can be generated by $\Rep^0(F^\times \bs G(F))$.
\end{lem}	
\begin{proof}
For any  character $\omega:\ F^\times\to\BC^\times$,  
  set $\Rep^0(D^\times(F),\omega):=\Rep(D^\times(F),\omega)\cap \Rep^0(D^\times(F))$. Then  by the classification of irreducible $D^\times(F)$-representations,
   for any irreducible $\pi \in \Rep(D^\times(F),\omega)$, there exists $\pi_0 \in \Rep^0(D^\times(F),\omega)$ 
   such that the semisimplification of $\pi_0$ is the direct sum of $\pi$ and another irreducible
   representation in $\Rep^0(D^\times,\omega)$.
   The lemma follows from this fact.
\end{proof}

\begin{lem}\label{lem-constancy}
 Fix $\sigma \in \mathrm{Tmp}(F^\times \bs M(F))$ for some Levi $M$ of $G$ and consider
		the unramified twisting family $\pi_\chi = I_P^G \sigma \otimes \chi \in \Rep^0(F^\times \bs G(F))$ with 
		$\chi \in \wh{F^\times \bs M(F)}^{\mathrm{un}}$. For any $T \in \CT$, the function
		\[\chi \mapsto c_{\pi_\chi}|_T\]
		is constant. In particular, the geometric multiplicity is constant for an unramified
		twisting family in the sense that the function
		\[\chi \mapsto m_\mathrm{geo}(\pi_\chi)\]
		is constant.
\end{lem}
\begin{proof}
	We have the following fact on the regularized character of a parabolic induction \cite[Proposition 2.7]{Wan21}.
	Let $\pi = I_P^G \sigma$ be a finite length  $G(F)$-representation induced from $P = MN$. 
	Then for any semi-simple $x \in G(F)$
	\[D^G(x)^{1/2} c_\pi(x) = \sum_{y \in \fX_M(x)} D^M(y)^{1/2} c_\sigma(y)\]
	where $\fX_M(x)$ is the set of representatives for $M(F)$-conjugacy classes
	of elements in $M(F)$ that are $G(F)$-conjugated to $x$.

	Let $T \in \CT$. If $\fX_M(t)$ is empty for any $t \in T(F)$, $c_{\pi_\chi}|_T = 0$ for 
	any $\chi$. Assume
        there exists $t \in T(F)$ such that $\fX_M(t)$ is nonempty, or equivalently, there is
	an embedding of $T$ into $M$. This happens exactly when
	\begin{itemize}
		\item $M = G$;
		\item $L = E \oplus F$, $T =  E^\times$  
		and $M = A_E \times D^\times$ where $A_E$ is the diagonal torus
	of $\GL_{2,E}$. In this case, 
	\[\fX_M(t) = \left\{ \left(\matrixx{t}{0}{0}{\bar{t}},t\right), \quad
	\left(\matrixx{\bar{t}}{0}{0}{t},t\right)\right\}.\]
	\end{itemize}
	To show the constancy of $c_{\pi_\chi}|_T$ for any unramified character $\chi$ on $M(F)$, 
	it is enough to prove that
	\[F^\times \bs T(F) \subset \left(F^\times \bs M\right)^0 = \bigcap_{\mu \in \mathrm{Rat}\left(F^\times \bs M\right)} \ker |\mu|\]
	where $\mathrm{Rat}\left(F^\times \bs M\right)$ is the group of rational characters on $M$. As $F^\times \bs T$ is anisotropic,
	$\mathrm{Rat}(F^\times \bs T) = 0$ so that for any $\mu \in \mathrm{Rat}\left( F^\times \bs M \right)$, $\mu(t) = 1$ for any $t \in T$.
\end{proof}

Now, as the both sides of the multiplicity formula is additive, by (1), we only need to consider representations
in $\Rep^0(F^\times \bs G(F))$, that is, unramified twists of tempered representations. 
By (2), the geometric multiplicity is constant for an unramified twisting family. Meanwhile,
it is known that the Euler-Poincar\'e number 
is constant for an unramified twisting family (see \cite[Theorem E(4)]{AS20} or \cite[Proposition 3.18]{CF21}
for a more general situation). Therefore, the multiplicity formula for tempered representations 
implies the formula for any irreducible representation.

\end{proof}

\begin{remark}
	In fact, the constancy of the geometric multiplicity holds in general. It can be proved
	similarly as the above special case Lemma \ref{lem-constancy} (Note that in general,
	any torus $T$ in the support $\CT(G,H)$ satisfies $T/Z_{G,H}$ is anisotropic. 
	See \cite[Definition 4.3]{Wan21}.)
\end{remark}


\begin{thebibliography}{XXXX}

	\bibitem{AS20}A. Aizenbud and E. Sayag,
	\emph{Homological multiplicities in representation theory of $p$-adic groups}. (English summary)
Math. Z. 294 (2020), no. 1-2, 451-469.
		\bibitem{Bum97}D. Bump, \emph{Automorphic forms and representations}
	Cambridge Studies in Advanced Mathematics, 55. Cambridge University Press, Cambridge, 1997. xiv+574 pp. ISBN: 0-521-55098-X
\bibitem{Cas95}W.A. Casselman, \emph{Introduction to the theory of admissible representations of $p$-adic reductive groups}, Draft 1995,
       available at \url{https://secure.math.ubc.ca/~cass/research/pdf/p-adic-book.pdf}
\bibitem{CF21} L. Cai and Y. Fan, \emph{Families of Canonical Local Periods on Spherical Varieties}, arXiv:2107.05921.
\bibitem{CS} K. Chan, G. Savin, {\em A vanishing Ext-branching theorem for $(\GL_{n+1}(F), \GL_n(F))$.} 
	Duke Math. J. 170 (2021), no. 10, 2237–2261.

\bibitem{Mat11} N. Matringe, \emph{
	Distinguished generic representations of $\GL(n)$ over $p$-adic fields}, Int. Math. Res. Not.
IMRN (2011), no. 1, 74–95.
\bibitem{NP20}M. Nori and D. Prasad, \emph{On a duality theorem of Schneider–Stuhler}, J. Reine Angew. Math., vol. 2020, no. 762, 2020, pp. 261-280. 
	\bibitem{Pra90}D. Prasad, \emph{Trilinear forms for representations of $\GL(2)$ and local $\epsilon$-factors.}
Compositio Math. 75 (1990), no. 1, 1-46.
	\bibitem{Pra92}D. Prasad, \emph{Invariant forms for representations of $\GL_2$ over a local field.} Amer. J. Math. 114 (1992), no. 6, 1317-1363.
	\bibitem{Pra18}D. Prasad, \emph{Ext-analogues of branching laws}. (English summary) Proceedings of the International Congress of Mathematicians Rio de Janeiro 2018. Vol. II. Invited lectures, 1367-1392, World Sci. Publ., Hackensack, NJ, 2018.
\bibitem{Wan17} C. Wan, \emph{A local trace formula and the multiplicity one
	theorem for the Ginzburg-Rallis model}, Ph.D. thesis, 2017. 
\bibitem{Wan21} C. Wan, \emph{On multiplicity formula for spherical varieties}, accepted by Journal of the European Mathematical Society.
\end{thebibliography}
	\end{document}